\documentclass[12pt]{article}

\usepackage[a4paper]{geometry}
\usepackage{amsmath,amssymb,amsthm,bbm,mathrsfs,url,xspace}
\usepackage[T1]{fontenc}
\usepackage[utf8]{inputenc}
\usepackage[vietnamese,english]{babel}
\usepackage[all]{xy}
\usepackage{array}
\usepackage{verbatim}
\usepackage{hyperref}
\usepackage{mathrsfs}
\usepackage{color}
\usepackage{tikz-cd}
\usepackage{authblk}

\newtheorem{thm}{Theorem}[section]

\newtheorem{lem}[thm]{Lemma}
\newtheorem{cor}[thm]{Corollary}
\newtheorem{prop}[thm]{Proposition}

\newtheorem{rmq}[thm]{Remark}
\newtheorem{exe}[thm]{Example}

\newtheorem{conj}[thm]{Conjecture}
\newtheorem{exercice}[thm]{Exercice}

\numberwithin{equation}{section}
\def\Tor{\mathrm {Tor}}
\def\Ext{\mathrm {Ext}}
\def\GL{\mathrm {GL}}

\def \Hom{\mathrm{Hom}}

\def\Sq{\mathrm{Sq}}

\def \rpinf {\mathbb{R}P^\infty}
\def \cpinf {\mathbb{C}P^\infty}
 \def\nil {\mathcal{N} il}
\def\unil {\calU/\mathcal{N} il}

\def\map{\mathrm{map}}

\def\rmH{\mathrm{H}}

\def\calA{\mathcal A}

\def\calF{\mathcal F}
\def\calG{\mathcal G}

\def\calK{\mathcal K}

 \def\calU{\mathcal U}

\def\bRP{{\mathbb RP}}
\def\bF2{{\mathbb F}_2}

\def \bZ{\mathbb{Z}}

\def \ra {\rightarrow}
\def \lra {\longrightarrow}
\def \hra {\hookrightarrow}
\def \era {\twoheadrightarrow}

\begin{document}

\title{On the mod-$2$ cohomology of some $2$-Postnikov towers}

\author{Nguy\~{\^{e}}n Th\'{\^{e}} C\textviet{\uhorn\`{\ohorn}}ng \thanks{Nguy\~{\^{e}}n Th\'{\^{e}} C\textviet{\uhorn\`{\ohorn}}ng,  Falculty of Mathematics, Mechanics and Informatics,
VNU University of Science, 334 Nguy\~{\^{e}}n Tr\~{a}i, Thanh Xu\^an, H\`a N\^oi, Vi\^et Nam} \and Lionel Schwartz \thanks{Lionel Schwartz, LAGA, UMR 7539 et LIA Formath Vietnam du CNRS, Universit\'e Paris Nord Av. J. B. Cl\'ement 93430 Villetaneuse, France}}
\maketitle

\abstract{The present note presents some results about the mod-$2$ cohomology, modulo nilpotent elements elements of
the fiber $E$ of a decomposable map $\psi : K(\bZ,2) \ra K(\bZ/2,p)$. This is more an announcement and a brief description of the tools that are used: Lannes' $T$ functor and the Eilenberg-Moore spectral sequence.}

\section{Introduction}


Let $E$ be the homotopy fiber of a map $\psi : K(\bF2,q) \ra K(\bF2,p)$ between two Eilenberg-Mac Lane spaces, the homotopy class $[\psi] $ belongs to $\rmH^q (K(\bF2,q), \bF2)$. If $\psi$ is a stable map the work of Larry Smith, David Kraines \cite{Sm1}  \cite{Sm2} \cite{Kr1} and others
 describes the mod-$2$ cohomology of $E$. The results of L. Smith  concern the fiber of an $H$-map between $H$-spaces. {Those of Kraines concern more specific examples}.  One of their main tool is the Eilenberg-Moore spectral sequence.

 To make notation shorter we will write $K_n$ for the Eilenberg-Mac Lane space $K(\bF2,n) $ in all the sequel. In all this note the mod-$2$ singular cohomology of a space $X$  will be denoted by $\rmH^*(X)$.

\medskip

In this  note we study the case where (the homotopy class of) $\psi : X \ra K_p$  is a decomposable element in the cohomlogy group $\rmH^p (X)$:  $\psi = \sum_i \,\, a_ib_i$, with $|a_i| >0$, $|b_i| >0$.
 One investigates what is possible to say on the cohomology of the homotopy fiber modulo the ideal of nilpotent elements.

After general results one focuses on the case $X=K_q$.

The main statements are:

\begin{thm}\label{mono} There is  a monomorphism:
	$$
	\frac{ ({\rmH^*(X)/(\psi^*(\iota_p)})}{Nil }\hookrightarrow  \frac{\rmH^*(E)}{Nil}
	$$
	where $Nil$ denotes the ideal of nilpotents elements.
\end{thm}

This a reformulation of  of \cite[Th\'eor\`eme 0.5]{LS2}

\begin{thm}\label{epi} There is  an epimorphism, modulo nilpotent elements:
	$$
\rmH^*(E) \ra \rmH^*(K_{p-1})
	$$
	This means that for any element $x \in \rmH^*(K_{p-1})$
	some power $x^{2^h}$, $h$ large enough, is in the image.
\end{thm}

Part of the results  are to some extent,   already contained in  \cite{Sc1}  and  \cite{Sc2}, but are made here  more precise.
The next questions are to decide how these two theorems "match" together and to describe the right hand side quotient more explicitly in some cases. As we have said we are going to discuss this modulo nilpotent elements.

The techniques described in the note makes the following conjecture very plausible. The notations are explained after the statement.

\begin{conj} \label{conjw}
(Weak form) \label{WF}	 As graded algebras one has:
	$$
	\sqrt {\rmH^*(E)} \cong   \sqrt {\rmH^*(X)/(\psi^*(\iota_p)}) \otimes \rmH^*(K_{p-1}) \; \; \; \; \; .
	$$
	\end{conj}
\begin{conj}	\label{conjs}
(Strong form)	\label{SF} As unstable algebras one has:
	$$
	\sqrt {\rmH^*(E)} \cong   \sqrt {\rmH^*(X)/(\psi^*(\iota_p)}) \otimes \rmH^*(K_{p-1}) \; \; \; \; \; .
	$$
\end{conj}

The algebra   $\sqrt {\rmH^*(Y)}$, for any space $Y$, denotes the quadratic closure of the quotient of the (graded) algebra $\rmH^*(Y)$ by the ideal of nilpotent elements $Nil$.
The unstable algebra $\sqrt {\rmH^*(Y)}$  is the smallest (graded) algebraic extension, in the sense of J. F. Adams and C. Wilkerson \cite{AW}, of $\rmH^*(Y)/Nil$ which contains all square roots.
\medskip

For example:

\begin{itemize}

\item $\sqrt{\rmH^*(K_p)} =\rmH^*(K_p)$;
\item
 consider $\rmH^*(\cpinf) \cong \bF2[u^2]$ then $ \sqrt{\rmH^*(\cpinf)}=\rmH^*(\rpinf)= \bF2[u] $;

\item  $\sqrt{K}
	\cong \bF2$for all  finite connected unstable algebra $K$ .
	\end{itemize}

	\bigskip

From now on we discuss examples.

The mod-$2$ cohomology of an elementary abelian $2$-group $V_n$ of rank $n$ denoted by $\rmH^*V_n$, is a graded polynomial algebra on $n$ generators of degree $1$,
$\bF2[x_1,\ldots, x_n]$, these generators may also be denoted also by $u,v,w,\dots$.

Let  $X$ be the Eilenberg-Mac Lane space $K_2$. Let $Q_i$, $i \geq 0$ be the Milnor operation of degree $2^{i+1}-1$ \cite{Mil}. One has \cite{Se}:
$$
\rmH^*(K_2) \cong \bF2 [\iota_2,Q_0(\iota_2),Q_1(\iota_2) ,\ldots , Q_i(\iota_2),\dots]  \; \; \; \; \; .
$$

We now describe  the right hand side algebra $\rmH^*(K_2)/(\psi^*(\iota_p))$.

\begin{prop} \label{case0} If  $[\psi]= \iota_2^t \in \rmH^*(K_2)$, $t>1$, then
$$
\sqrt{\rmH^*(K_2) \otimes_{\rmH^*K_p} \bF2} \cong \bF2
$$
\end{prop}

\begin{prop} \label{case1} If  $[\psi]\in \rmH^*(K_2)$ is a nontrivial  {\it monomial} in the polynomial generators of $\rmH^*(K_2)$, involving at least one $Q_i(\iota_2)$, then
$$
\sqrt{\rmH^*(K_2) \otimes_{\rmH^*K_p} \bF2 }\cong \rmH^*(B\bZ/2) \cong \bF2[u]
$$
\end{prop}
\begin{prop} \label{case2} If  $[\psi]= \alpha  d_2$, with $d_2=\iota_2 Q_0(\iota_2)+ Q_1(\iota_2)$, $\alpha$ any nontrivial monomial,  then
$$
\sqrt{\rmH^*(K_2) \otimes_{\rmH^*K_p} \bF2 }\cong D(2)
$$
where $D(2)=\rmH^*(V_2)^{\GL_2(\bF2)}$ is the mod-$2$ Dickson algebra.

\end{prop}

\begin{prop} \label{case3}If  $[\psi]=\alpha h_2$ with $h_2=\iota_2^2Q_1(\iota_2)+ Q_0(\iota_2)^3+ \iota_2^3 Q_0 \iota_2$, $\alpha$ any nontrivial monomial,  then
$$
\sqrt{\rmH^*(K_2) \otimes_{\rmH^*K_p} \bF2 }\cong H_2
$$
where $H_2 \subset \rmH^*(B(\bZ/2 \times \bZ/2)) \cong \bF2[u, v]$, $ \vert u \vert  = \vert v \vert =1$ is the  unstable subalgebra generated by $$\{uv, uv(u+v), uv(u^3+v^3),\ldots, uv(u^{2^n-1}+v^{2^n-1}), \dots \}  \;\;\;\; .$$
This is also the subalgebra of $\bF2[u, v]$ generated by $$w_2, w_1w_2, w_1^2 w_2,\ldots, w_1^nw_2, \ldots where $$ $w_1,w_2$ are the universal Stiefel-Whitney classes.
\end{prop}

\begin{prop} \label{case4}If  $[\psi]=\alpha d_2h_2$, $\alpha$ any monomial,  then
$$
\sqrt{\rmH^*(K_2) \otimes_{\rmH^*K_p} \bF2 }\cong \sqrt{ M_2}
$$
where $M_2 \subset \rmH^*(B(\bZ/2 \times \bZ/2))^{\times 2} \cong \bF2[u, v]^{\times 2}$,  is the unstable subalgebra which is the fiber product of $D(2)$ and $H_2$ over $\bF2[u]$ over the nontrivial morphisms $D(2) \ra \bF2[u]$ and $H_2 \ra \bF2[u]$.
\end{prop}

In fact all of the preceding results (\ref{case0} to \ref{case3}) hold  without the assumption  "nontrivial momomial". This hypothesis is used in the sequel.

\begin{exercice} \label{case5} Find classes $[\psi] \in \rmH ^p(K_2)$ such that $\sqrt{\rmH^*(K_2)/(\psi^*(\iota_p)}$ is the unstable subalgebra
	of $\bF2[u,v,w]$ generated by $uv+w^2$.
\end{exercice}

\begin{prop} \label{conjecture-ex} In all of these examples the weak form of the conjecture  olds. The strong form also holds except in Proposition  \ref{case4}, in which case it  is true if $p-1$ is odd. \end{prop}

 We will use  the $T$-technology of J. Lannes \cite{LaT},  \cite{Sc2} and also the Eilenberg-Moore spectral sequence
\cite{Rec}, \cite{Sm2}. We make the necessary recollections in Section $1$.  In Section 2 we discuss the right hand side quotients in \ref{conjexe}. In Section 3 we prove Theorems \ref{mono} and \ref{epi}.
In Section 4 we prove \ref{conjecture-ex}. The methods developed is very much likely to extend to solve Conjecture \ref{conjw}.

Tihis work will be continued in order to generalize the results, in two directions. The first direction is to extend the results on the reduced part of the cohomology to some more general examples. The second one (and this is the reason for Theorem 4.5) is to extend the results to the $1$-nilpotent part of the cohomology. It would be also worth, harder but possible, to try to get some informations about the space of maps $\map(BV \wedge \bRP^2, E)$, this would give insights on the $2$-nilpotent part of the cohomology of $E$.

\bigskip

{\bf Acknowledgement: }{\it This work was initiated originally during the first months of the PhD thesis of the first author who made a talk about it at the congress of the learned societies in mathematics of France and Vietnam in Hu\'e (2012). It started again and grew while the second author was visiting VIASM in Hano\"{i} in November-December 2019. This research is funded by the Vietnam National University, Hanoi (VNU) under project number QG.20.28.}

\section{Recollections about unstable modules and algebras.}

We recall here some notations  and facts about unstable modules and unstable algebras over the Steenrod algebra.

One denotes as usual $\calU$ the abelian category of unstable modules over the Steenrod algebra $\calA_2$, and by $\calK$ the category of unstable algebras. An unstable module $M$ is a (graded) module over the (graded) Steenrod algebra
which satisfies $\Sq^i(m)=0$ for any $m \in M$ such that $\vert m \vert <i$.  An unstable algebra is a (graded) algebra which is also an object in $\calU$, the two structures being compatible : the Cartan formula and the restriction axiom hold.

As usual,  $\Sigma$ denotes the suspension functor in the category $\calU$. An unstable module $M$ is a suspension if and only if the map $\Sq_0 : M \ra M$, $m \mapsto \Sq^{\vert m \vert}(m)$ is trivial. Let $\nil_i$ be the Serre classe in $\calU$ generated by $i$-th suspensions : it is closed by sub-objects, quotients, extensions and colimits \cite{Sc1}, \cite{Sc2}; $\nil_1$  is more commonly denoted by $\nil$. In an unstable algebra $K$  the ideal consisting of all nilpotent elements is stable
 under the action of the Steenrod algebra, and is the largest submodule of $K$ in $\nil$.

 Details can be found, for example,  in \cite{St},  \cite{LaT}, \cite{Sc1}, \cite{LZens}, \cite{FFPS}.
 \medskip

We will say that a morphism
$f : M \ra N$  of unstable modules is an $F$-monomorphism (resp. epimorhism) if  the kernel (resp. cokernel) are  in $\nil$. The definition restricts (and is more classical) to unstable algebras:    $f : K \ra L$, is an $F$-monomorphism if any element in the kernel
is nilpotent, and an $F$-epimorphism if for any $x \in L$ there exists $n$ such that $x^{2^n}$ is in the image.

Being an $F$-monomorphism and an $F$-epimorphism at the same time is called an $F$-isomorphism.

As the category $\nil$ is a localizing subcategory \cite{Gab}, in the quotient category $\calU/\nil$ an $F$-isomorphism becomes an isomorphism. The canonical functor $\calU \ra \unil$  has a right adjoint and there is a localisation functor $\ell: \calU \ra \calU$. The unstable module $\ell(M)$ is $\nil$-closed \cite{Gab}.  If $K$ is an unstable algebra, then so is $\ell(K)$. Moreover the unit of the adjunction   $\calK $: $K \ra \ell(K)$ is a morphism of unstable algebras \cite{BZ}.

The ideal of nilpotent elements
in $\ell(K)$ is trivial and $\ell(K)$ is quadratically closed. This last property means that one cannot find $J \supset \ell(K)$, an element $x \in J \setminus \ell(K)$, with $x^2 \in \ell(K) \setminus \{0\}$ \cite{LZens}
For this reason we are going to denote $\ell(K)$ by $\sqrt K$.

\medskip
Now, recall the functor $f: \calU \ra \calF$,  where $\calF$ is the category of functors from the category of finite dimensional l$\bF2$- vector spaces
to the category of  $\bF2$- vector spaces. The functor $f$ sends an unstable module $M$ to the functor $V \mapsto \Hom_\calU(M,\rmH^*(V))^\#$. It is shown
in \cite{HLS} that $f$ induces an equivalence  (also denoted by $f$) $f : \calU/\nil \ra \calF_\omega$ of  $ \calU/\nil $ with the subcategory of
analytic functors $\calF_\omega$.

The dual considered above  is a continuous one. Indeed,
$\Hom_\calU(M,\rmH^*(V))$ has a natural profinite structure (\cite{HLS}). However in the cases we  consider all vector spaces  are  finite dimensional . Thus
we will ignore profinite structures.

Recall that Lannes' functor $T_V: \calU \ra \calU$ , left adjoint to  $M \mapsto \rmH^*(V) \otimes M$,   is exact and commute with tensor product \cite{LaT}.  Following
\cite{HLS} $f(M)(V)$  is defined to be $T_V(M)^0$.

\bigskip

As in (\cite{HLS}), let $g : \calK \ra \calG$, where $\calG$ is the category of  contravariant functors from the category of finite dimensional $\bF2$ -vector spaces   to the category of sets, be defined by $K \mapsto
\{V \mapsto \Hom_\calK(K,\rmH^*(V))\}$. It
 induces
an equivalence (also denoted by $g$) $g : \calK/\nil \ra \calG_\omega$ to  the category of set valued analytic functors.

Recall that  Lannes' "linearization theorem" writes as:
$$
f(K)(V) \cong \bF2^{\Hom_\calK(K,\rmH^*(V)) }
$$
where the right hand side denotes the Boolean algebra of set maps from $\Hom_\calK(K,\rmH^*(V))$ to $\bF2$.

Again  one should need profinite structures on sets, but again here sets are finite.

Let $U: \calU \ra \calK$  be the Steenrod-Epstein functor: the  universal enveloping algebra \cite{Sc2}. Recall (\cite{Se}) the isomorphism of unstable algebras  $\rmH^*(K_n) \cong U(F(n))$ where the free unstable modules $F(n)$ is   characterized by $$\Hom_\calU(F(n),M) \cong M^n \; \; \; \; \; .$$ We will also use later the fact the functor $T_V$ commutes with $U$: $T_V(U(M)) \cong U(T_V(M))$ and the following isomorphisms (which can be deduced rom part 1 of \cite{Sc2} and of \cite{FFPS}):
\begin{eqnarray}
T_V(F(p)) \cong \bigoplus_{0 \leq i \leq p} \Gamma^i(V) \otimes F(p-i)
\end{eqnarray}
where $\Gamma^i \in \calF$ is $i$-th divided power functor. And as $\rmH^*(K_p) \cong U(F(p))$ one gets:
\begin{eqnarray}
T_V(\rmH^*(K_p)) \cong  \bigotimes_{0 \leq i \leq p} \rmH^*(K( \Gamma^{i}(V^\#), p-i)) \; \; \; \; \; .
\end{eqnarray}

\section{Examples:  subfunctors of $S^2$ and quotients of $\rmH^*(K_2)$}

Next consider the  examples alluded to in the introduction:



\begin{exe}The set valued functor    $g(\rmH^*(K_2))$:
	$$
	V \mapsto \Hom_\calK(\rmH^*(K_2), \rmH^*V) \cong S^2(V^\#)
	$$
is filtetered by the rank $r$ of the associated  bilinear  form (which is an even integer $r=2k$) by subfunctors:
$$
S^{2,0}
 \subset S^{2,0}
 \subset S^{2,2}
 \subset S^{2,4}
 \subset
 \ldots
 S^{2,2k}
 \subset
\ldots$$

 A subfunctor $G$ is said to be generated in dimension less than $d$ if any element in $G(V)$ is in the image of $f^*$, for some
linear map $f :V \ra V_d$.
 There is a diagram of subfunctors:

$$
\xymatrixcolsep{0.8em}\xymatrix{
	&				&	\left\langle \sum_{1 \leq i \leq r-1}u_iv_i+u_r^2+u_rv_r+v_r^2 \right\rangle  \ar[dr]^{\cup}&		\\
	&	S^{2,2(r-1)} 	 \ar[dr]^{\cup}	\ar[urrl]^{\cup} &		&	\left\langle\sum_{1 \leq i \leq r}u_iv_i+w^2 \right\rangle	\ar[r]  &	S^{2,2r}	  	  \\		&                               &		S^{2,2(r-1)}  \ar[ur]^{\cup}&				\\
}
$$

\medskip

In this diagram quadratic forms generating the subfunctors  are written as usual. In rank $2r$ and dimension $2r$ there are two non isomorphic forms, one of Arf invariant $1$, the other of invariant $0$. In dimension $2r+1$ and  rank $2r$ there is the form $\left\langle  \sum_{1 \leq i \leq r}u_iv_i+w^2 \right\rangle$.

\end{exe}

\begin{exe} \label{casparticulier}
The subfunctors generated in dimension less than $2$ are:
\begin{itemize}
\item the subfunctor $S^{2,0}$ itself is filtered by the constant subfunctor generated by $S^*(0)$, and the subfunctor generated by the form $\left\langle  u^2  \right\rangle \in S^2(\bF2^\#)$.
\item the subfunctor $S^{2,2}$ contains as subfunctor the one generated by the form $\left\langle  u^2+uv+v^2  \right\rangle \in S^2((\bF2^{\oplus 2})^\#)$,
\item the one generated by the form $\left\langle uv  \right\rangle \in S^2((\bF2^{\oplus 2})^\#)$;
\item the subfunctor generated by both forms $\left\langle uv, u^2+uv+v^2  \right\rangle \in S^2((\bF2^{\oplus 2})^\#)$;
\item $S^{2,2}$ is generated by the form  $\left\langle u^2+uv+v^2+w^2  \right\rangle \in S^2((\bF2^{\oplus 3})^\#)$.
\end{itemize}

\end{exe}

\medskip

We come back now  to  $X=K_2$.
The quotients (in $ \calK/\nil $)  of $\rmH^*(K_2)$ are classified by subfunctors  of $V \mapsto S^2(V^\#)$.
The correspondance:
 $$
 \calK/\nil \ra \calG_\omega,  \, \, \, \, K \mapsto \{V \mapsto \Hom_\calK(H, \rmH^*(V))\}
 $$
 is an equivalence of categories. Therefore there is a one-to-one correspondance between subfunctors of $S^2(-)^\#$ and the quotients in $\calK/\nil$
 of $\rmH^*(K_2)$.Recall:
$$
\rmH^*(K_2) \cong \bF2 [\iota_2,Q_0(\iota_2),Q_1(\iota_2) ,\ldots , Q_i(\iota_2),\dots]
$$

\medskip

We compute explicitly the quotients $\sqrt{\rmH^*(K_2)/(\psi^*(\iota_p))}$ mentioned in Example \ref{casparticulier}. We do not claim that they are in any sense the best examples and this deserves further investigations.

\medskip

As $\rmH^*(V) \cong S^*(V^\#)$, a (homotopy class of) map: $\psi : K_p \ra K_2$ determines a natural transformation:
$$
\psi_* : S^2(-) ^\#\ra S^p(-)^\#
$$
 It determines a subfunctor in $S^2$ as follows: it consists of the set of elements of $S^2(V^\#)$ which are sent to $0$
by this class:
$$
\{s \in S^2(V^\#) \,  | \,\psi_* (s)=0 \}
$$

 The map $\psi_* $ is just a cohomology operation: a polynomial in the Steenrod powers. The computation is done using the usual rules with the Steenrod algebra.

 \medskip

 We now illustrate some examples.

\begin{itemize}
\item Let $[\psi]$ be  $\iota_2^t$, for some integer $t>1$. In this case
$$
\psi_* : S^2(-) ^\#\ra S^p(-)^\#
$$
is injective. The kernel of this morphism is the constant functor on one point. The corresponding quotient is $\bF2$.

\item Let $[\psi]$ be  $\alpha Q_i(\iota)$ for some monomial. As In the case of the  subfunctor generated by $\left\langle u^2 \right\rangle$  any monomial involving one $Q_i(\iota_2)$ works because it is zero on $u^2$ and non zero on $uv$ and $u^2+uv+v^2$:
$$
Q_i(u^2)=0
$$
$$
Q_i(uv)=uv(u^{2^{i+1}-1}+v^{2^{i+1}-1})
$$
$$
Q_i(u^2+uv+v^2)=uv(u^{2^{i+1}-1}+v^{2^{i+1}-1})
$$
the kernel of $\psi_*$ is generated by $\left\langle u^2 \right\rangle$. The corresponding quotient is therefore $\bF2[u]$, whence Proposition  \ref{case1}.

\item Similarly if $[\psi]=\alpha d_2$, and $d_2=(\iota_2 Q_0(\iota_2)+ Q_1(\iota_2))$, $\alpha$ any monomial,
 a direct
computation shows that the class
$\psi_*$ is generated by  $\left\langle  u^2+uv+v^2 \right\rangle$. The corresponding quotient is the Dickson algebra generated by the Dickson invariants $u^2+uv+v^2, uv(u+v)$. Proposition \ref{case2} follows.

The other cases follow. Proposition \ref{conjexe} will be proved in section 4. The subfunctors of Example \ref{casparticulier} classify (modulo $\nil$) the quotients of transcendance degree at most $2$ and $3$ of $\rmH^*(K_2)$,
$D(2)$ and $H_2$ are the only quotients of $\rmH^*(K_2)$ which are integral domain and of transcendance degree $2$. This follows from the classification of quadratic forms over $\bF2$ (described above) and
\cite{HLS}.

\end{itemize}

\begin{prop} \label{conjexe} Let \(\psi:K_2\to K_p\) be one of the maps in Propositions \ref{case0} to \ref{case3}  \(E\) the homotopy fiber of \(\psi\), then there is an isomorphism of unstable algebras:
	 \[
	 \sqrt{\rmH^*(E)}\cong\sqrt{\rmH^*(K_2)\otimes_{\rmH^*K_p} \bF2 }
	 \]
	 For Proposition \ref{case4}
this holds if $p-1$ is odd.
\end{prop}

This will be proved in the last section.

\section{Sets of homotopy classes of maps, applications}

\bigskip

Let us start with a classical remark about homotopy classes of maps (\cite{Hat}, Adittionnal topics : Base points and Homotopy 4A).

\begin{prop} Let $X$ and $Y$ be two connected pointed $CW$-complexes then the natural map from the set of pointed homotopy classes from $X$ to $Y$ to homotopy classes :
$$
[X,Y]_* \ra [X,Y]
$$
is always surjective. It is a bijection if either $X$ is $1$-connected or $X$ is $H$-space.
\end{prop}

Recall that for an elementary abelian \(2-\)group \(V\), \(BV\) is the classifying space of \(V\). The fibration \(E\ra X \buildrel\psi \over \ra K_p\) yields a short exact sequence of pointed sets:

\begin{eqnarray}
[\Sigma BV,E ]_* \ra [\Sigma BV,X]_* \ra [\Sigma BV,K_p]_* \ra  \\
\ra[BV,E]_* \ra [BV, X ]_* \ra [BV,K_p]_*  \nonumber
\end{eqnarray}
If $X$ is $1$-connected it is equivalent to :
\begin{eqnarray}
[\Sigma BV,E ] \ra [\Sigma BV,X] \ra [\Sigma BV,K_p] \ra \\
\ra [BV,E] \ra [BV, X ]\ra [BV,K_p]  \nonumber
\end{eqnarray}

Where:

\begin{itemize}
	\item The first 3 terms are groups and the sequence is exact in the usual sense.
	\item The last three terms are pointed sets, the composite morphisms are trivial.
	\item  Exactness at $[BV, X ]_*$ means that the (homotopy class) of a map $f : BV \ra X$ lifts to $E$ if and only if the composite $BV \ra X \ra K_p$ is trivial.
	\item Exactness at $[BV, E ]_*$ means that the (homotopy class) of a map $f : BV \ra E$ lifts to $[\Sigma BV,K_p] $ if and only if the composite $BV \ra E \ra X$ is trivial.
	\item Moreover $[\Sigma BV,K_p]_*$ acts on $ [BV,E]_* $, and
	two homotopy classes $f,g : BV \ra E_d$ which have the same image in $[BV,X]_*$   are in the same orbit under the action of $[\Sigma BV,K_p] $.
	\item An element of $[\Sigma BV,K_p] $
	which acts trivially on  $ [BV,E]_* $ comes from $[\Sigma BV,X] $.
\end{itemize}
\bigskip

We are going to use the following two theorems:
\begin{thm} [\cite{LaT}] Let $X$ be a connected  CW-complex such that $\rmH^*(X)$ is of finite dimension in any degree and that $\pi_1(X)$ is a finite $2$-group. Then the natural map
$$
 [BV,X] \ra \Hom_\calK (\rmH^*(BV), \rmH^*(X))  ,\;\;\;\;  f \mapsto f^*
$$
is a bijection
\end{thm}

The second theorem below is outside of the main stream of the note. It is there to show that this type of strategy also allows  to get informations on the ideal $Nil \subset \rmH^*(E)$. This will be done done in future articles.

\begin{thm} \label{epiLS} \cite{LS1}
The natural map
$$
  [\Sigma BV,X] \ra \Hom_\calK (\rmH^*(X), \rmH^*(\Sigma BV)), \;\;\;\; f \mapsto f^*
$$
is a surjection
\end{thm}

 For an augmented unstable algebra \(K\), let \(\tilde{K}\) be the ideal of augmentation. Denote by
	 \[
	 Q(K):=\frac{\tilde{K}}{\tilde{K}\cdot \tilde{K}}
	 \]
	 the unstable module of indecomposable elements. For any unstable algebra  \(K\),  \(Q(K)\)  is a suspension.   Remark that this is true without connectivity hypothesis on $K$ because $K^0$ is a Boolean algebra.
One has:
	 \[
	 \Hom_\calK (\rmH^*(X), \rmH^*(\Sigma BV)) \cong \Hom_\calU (\Sigma^{-1}Q(\rmH^*(X)), \tilde \rmH^*(BV))
	 \]

\medskip

Thanks to  4.2 and 4.3, sequences 4.1 and 4.2 yield the following exact sequences:
 \vspace{-.5cm}
	 \[
	 \begin{tikzcd}
	 \Hom_\calK( \rm\rmH^*(E),H^*(\Sigma BV))\arrow[d]& \Hom_\calK( \rm\rm H^*(E),H^*(BV))\arrow[d]\\
	 \Hom_\calK( \rm\rmH^*(X),H^*(\Sigma BV))\arrow[d]& \Hom_\calK( \rm\rmH^*(X),H^*(BV))\arrow[d]\\
	 \Hom_\calK( \rm\rmH^*(K_p),H^*(\Sigma BV))\arrow[ruu, controls={+(5,0) and +(-5,1.5)}]& \Hom_\calK( \rm\rmH^*(K_p),H^*(BV))
	 \end{tikzcd}
	 \]

Note that we have the following isomorphisms:
$$
[\Sigma BV,Y]_* \cong  [BV, \Omega X]_* \cong \Hom_\calK(\rmH^*(\Omega X), \rmH^*(BV))
$$

\bigskip

Now we use the hypothesis that the class $\psi$ is decomposable. The following is classical.

\begin{lem}
	If $\psi \in \rmH^p(X)$  is decomposable the class $\Omega \psi \in \rmH^{p-1}(\Omega X)$ is trivial.  Consequently  $\Omega E$ is homotopically equivalent to $\Omega X \times K_{p-2}$
\end{lem}

\begin{proof}

One has a commutative diagram :
$$
\xymatrix{
	&		\Sigma \Omega X \ar[d]^{ev}	\ar[rr]^{\Sigma\Omega \psi}		&		    &	\Sigma \Omega K_p \cong \Sigma K_{p-1}	\ar[d]^{ev}\\
	&	X		\ar[rr]^{\psi}		&	&	K_p	&   &		\\
}
$$
Denote by $\iota_p$ the generator of $\rmH^p(K_p)$. The evaluation map $ev$ sends $(t,\omega)$, $t \in[0,1]$,
$\omega \in \Omega X$
(resp. $\Omega K_p$) to $\omega(t)$. As all nontrivial cup-products are trivial in a suspension
$\psi \circ ev \in \rmH^p(\Sigma\Omega X) $ is trivial. By commutativity of the diagram $ev \circ \Sigma \Omega \psi$  is trivial. As $eV^\#(\iota_p) = \Sigma \iota_{p-1}$ this implies that $\Sigma (\Omega \psi)^*(\iota_{p-1}) =0$,  and
$(\Omega\psi)^*(\iota_{p-1} )=0$

\end{proof}

As $\Omega \psi \sim *$ one gets an inclusion
$$
\Hom_\calK(\rmH^*(K_{p-1}), \rmH^*(BV))\hookrightarrow \Hom_\calK(\rmH^*(E), \rmH^*(BV))
$$

In other words, we have an injection \(g(K_{p-1}) \ra g(\rmH^*(E))\). Because of the equivalence of categories \(\calK/\nil\simeq \calG_{\omega}\), we get an \(F-\)epimorphism \(\rmH^*(E) \ra \rmH^*(K_{p-1})\), whence Theorem 1.2.  Thus,  $\rmH^*(E) \ra \rmH^*(K_{p-1})$ is an $F$-epimorphism:
  a certain power $\iota_{p-1}^{2^h}$ of the fundamental class is in the image of the morphism, however, to find a representative may be hard.


\bigskip
The following result is a consequence of Theorem  \ref{epiLS}, it is mentionned for future investigations:

\begin{prop} Let $E$ be the homotopy fiber of a non homotopically  trivial   decomposable map $\psi : X \ra K_p$.  There is a monomorphism of set valued functors:
$$f(\Sigma^{-1}Q(\rmH^*E))/\nil)  \hookrightarrow S^{p-2} \oplus g(\rmH ^*(\Omega X))
 \;\;\;\;\; .$$

\end{prop}

It follows from  \ref{epiLS} that one has a surjection:
$$
 [\Sigma BV,E] \twoheadrightarrow \Hom_\calK (\rmH^*(E), \rmH^*(\Sigma BV)) \cong \Hom_\calU(\Sigma^{-1} Q(\rmH^*(E), \tilde \rmH^*V)  \;\;, $$
as  $$  [\Sigma BV,E]\cong [BV, \Omega E]\cong \Hom_\calK(\rmH ^*(\Omega E), H^*V)
 $$
and as we have isomorphisms
	 	\[
	 	\rmH ^*(\Omega E) \cong \rmH ^*(\Omega X)) \otimes \rmH^*(K_{p-2})  \;\;\;\;\;,
	 	\]
	 	the inclusion follows.



\section{The Eilenberg-Moore spectral sequence}

We are now going to prove  Proposition \ref{conjexe}. The methods are very much likely to adapt to prove Conjecture \ref{conjw}.
We will use the Eilenberg-Moore spectral sequence.
Recall various facts about the spectral sequence, references are \cite{Rec}, \cite{Sm2} and \cite{Dw}.
Let $F \ra E \ra B$ be a fiber sequence of pointed spaces,  assume that $\pi_1(B)$ is a finite $2$-group. The spectral sequence
has the following properties :
\begin{enumerate}
	\item  There is a deceasing filtration on $\rmH^* (F)$: $$\ldots \supset F_{-s}\supset \ldots \supset F_{-1} \supset F_0$$ by $\calA_2$-modules;
	\item  The $E^2_{s,*}$  page is isomorphic to $\Tor^{-s}_{\rmH^*(B)}(H^*(E),\bF2)$, $s \geq 0$;
\item  the differentials are of bidegree $(r,1-r)$ and are $\calA_2$-linear;
\item the spectral sequence is a multiplicative one, the product being given by the shuffle product;
\item  it converges towards $\rmH^*(F)$: the $k$-th term of the graded object associated to the filtration of $\rmH^k(F)$ is
 $\Sigma^{s}E^\infty_{s,-s+k}$;
\item $\Tor^{-s}_{\rmH^*(B)}(H^*(E),\bF2)$ is $(s-1)$-connected and   in $\nil_{s}$ (\cite{Sc2} Theorem 6.4.1).
\end{enumerate}

\medskip
About the convergence, it is obvious when  $B$ is $1$-connected, there is a vanishing line of slope $2$ : everything  below the line $2u+v=0$ is trivial and the filtration
of $\rmH^*(F)$ in a given degree is finite. For the general case see \cite{Dw}.

About (vi), $\Tor^{-s}_{\rmH^*(B)}(H^*(E),\bF2)$ is not in general an $s$-suspension, even if it is true for $s=1$ and $p=2$ (this is not true if $p>2$).  Being in \(\nil_{-s}\) follows from the fact that if \(M\) is a module over an unstable module \(K\), such that  the action satisfies the Cartan formula, then \(\Tor^{s}_{K}(M,\bF2)\) is \((-s-1)\)-connected (see \cite{Sc2} Lemma 6.4.3).

It follows from (ii) and (vi) that \(d_r(E^2_{s,*}) \in \nil_{-s-r+1}\).

Applications of these results are given  in \cite{Sc1}, see for example Theorem 8.7.1 and Proposition 8.7.7.

\medskip

It is worth and it will be useful to give some more informations on the construction of the spectral sequence. It follows from \cite{Rec} (see Section 3, the construction of the space $X^1$) that the $-1$ step of the
filtration on $\rmH^*(F)$ is, up to suspension, the image of $\rmH^*(X^n)$, is defined by the diagram of cofibrations below:
$$
\xymatrix{
	&F	\ar[r]  \ar[d]    &  \bullet  \ar[r]    \ar[d]         & \Sigma F \ar[r]  \ar[d]
	& \Sigma F  \ar[d] \\
	&E/\emptyset \ar[r]^{diag}&	E/\emptyset \wedge B \ar[r] &	X^1       \ar[r]   & \Sigma E/\emptyset	\\
}
$$

In our case $E=K_2$ and $B=K_p$.

This works as well when replacing all spaces $Z$ in this diagram by the corresponding mapping spaces $\map(BV,Z)$ (\cite{FS}).

\bigskip

In order to prove  \ref{conjexe}  one computes the associated graded object of $f(\sqrt {\rmH^*(F)})$ with respect  to the Eilenberg-Moore filtration.  Part (a) of the proposition follows directly. For part (b) one needs to control extensions.\begin{prop}
The Eilenberg-Moore filtration on $\rmH^*(F)$ induces a filtration on the functor $f(\rmH^*F) \cong f (\sqrt{\rmH^*F})$. The graded functor associated
to this filtration is a (graded) subquotient of the graded functor:
  $$
						 \left\{ T _{ V } \left( \Tor_{ \mathrm{H}^{ * }\left( B \right) }^{ - s }\left( \mathrm{H}^{ * }\left( E \right) , \mathbb{F} _{ 2 } \right) \right) ^{ s } \right\} _{ s \geq 0 }  \; \; \; \; \;.
						$$

(b) In our cases the spectral sequence degenerates at $E^2$. Thus  the associated  graded functor is equal to the preceding one and isomorphic to:

$$
 \left\{(\bF2^{\Gamma^2(V^\#)} \otimes_{{\bF2^{\Gamma^{p}(V^\#)} }} \bF2) \otimes \Lambda^s(\Gamma^{p-1})\right\}_{s \geq 0} \;\;\;\; .
$$
This isomorphism respect the multiplicative structure.
In terms of unstable modules this corresponds to:
$$
\sqrt {\rmH^*(X)/(\psi^*(\iota_p)}) \otimes \Lambda^*(F(p-1))
$$

\end{prop}

\bigskip

We come back to the  case   of a decomposable map $\psi : K_2 \ra K_p$.
One  needs  to compute: $
T_V(\Tor^{-s}_{\rmH^*(K_p)}(\rmH^*(K_2),  \bF2))^s
$.
Recall that:
$\Tor^{-s}_{\rmH^*(K_p)}(\rmH^*(K_2), \bF2)$ is  the $-s$-th cohomology of the reduced bar complex
\begin{eqnarray}
\ldots \ra \rmH^*(K_2)\otimes \tilde \rmH^*(K_p)^{\otimes s} \ra \rmH^*(K_2)\otimes \tilde \rmH^*(K_p)^{\otimes s-1} \ra \ldots \ra\rmH^*(K_2)
\end{eqnarray}

 As the functor \(T_V\) is exact and commutes with tensor products, applying it to the sequence 5.1 yields a complex whose \(-s\)-th cohomology is $$\Tor^{-s}_{T_V(\rmH^*(K_p))}(T_V(\rmH^*(K_2)),\bF2)$$(see\cite{Sc2}).


This complex
simplifies because $\rmH^*(K_p)$ is an unstable algebra (see the proof of 6.4.3 in \cite{Sc2}). Recall formulae (2.1) and (2.2):
$$
T_V(F(p)) \cong \bigoplus_{0 \leq i \leq p} \Gamma^i(V) \otimes F(p-i)
$$
and as $\rmH^*(K_p) \cong U(F(p))$ one gets:
$$
T_V(\rmH^*(K_p)) \cong  \bigotimes_{0 \leq i \leq p} \rmH^*(K(\Gamma^{i}(V^\#), p-i))
$$

The Boolean of maps from a  set $X$ to $\bF2$ is denoted by  $\bF2^{X}$,   $\overline{\bF2^{X} }$ denotes the ideal of maps taking the value $0$ at a chosen base point; in our case as the sets are vector spaces one chooses $0$ as base point.

The evaluation map at the base point
$\bF2^{X} \ra \bF2$ makes $\bF2$ a projective module over $\bF2^{X}$.
As $T_V(\rmH^*(K_p))^0 \cong\bF2^{\Gamma^p(V^\#)}$, it makes
$$L_p = \bigotimes_{1 \leq i \leq p} \rmH^*(K(p-i, \Gamma^{i}(V^\#))$$ a projective module over
$$T_V(\rmH^*(K_p)) \cong  \bigotimes_{0 \leq i \leq p} \rmH^*(K(\Gamma^{i}(V^\#), p-i) \;\;\;\; .$$  Thus:

 \begin{lem}The reduced bar resolution of $\bF2$ over $L_p$ is also a projective resolution
 of $\bF2$ over $T_V(\rmH^*(K_p))$.
  \end{lem}

  Define the unstable algebra $L_2$ by $L_2 =T_V(\rmH^*(K_2))\otimes_{\bF2^{\Gamma^p(V^\#)}} \bF2$.  It follows from the preceding lemma that
$T_V(\Tor^{-s}_{\rmH^*(K_p)}(\rmH^*(K_2), \bF2))$ is also   the $-s$-th homology of the reduced bar complex:
\begin{equation}\label{redbar}
       \cdots \ra L_2\otimes \bar L_p^{\otimes s} \ra L_2\otimes \bar L_p^{\otimes s-1} \ra \cdots \ra L_2 \ra 0
     \end{equation}

 In degree $0$ this is :

$$
\ldots \ra 0 \ra L^0_2 \cong \bF2^{\Gamma^2(V^\#)} \otimes_{{\bF2^{\Gamma^{p}(V^\#)} }} \bF2 \ra 0
$$
One has

	$$
	 L_2 \cong 	L_2^0 \otimes \rmH^*(K_2) \otimes \rmH^*(K(V^\#,1)) \;\;\;\; ,
	$$
 thus
     \[
       \begin{array}{cll}
       L^1_2 &\cong& V^\sharp \otimes L^0_2,\\
       L^1_p &\cong& \Gamma^{p-1}(V^\#),\\
       L^2_p &\cong& S^2(\Gamma^{p-1}(V^\#)),
       \end{array}
     \]
     in degree $s$, place $-s$, Complex \eqref{redbar} becomes:
     \[
     \begin{tikzcd}
     0\arrow[d]\\
     L_2^0 \otimes \Gamma^{p-1}(V^\sharp)^{\otimes s}\ar[d]\\
     V^\sharp \otimes L_2^0  \otimes \Gamma^{p-1}(V^\sharp)^{\otimes s-1}
     \bigoplus
     L_2^0 \otimes
     \left(\bigoplus_{i} \Gamma^{p-1}(V^\sharp)^{\otimes i} \otimes S^2(\Gamma^{p-1}(V^\sharp)) \otimes \Gamma^{p-1}(V^\sharp)^{\otimes s-3-i}\right)
     \end{tikzcd}
     \]

\medskip
We will show that the morphism is trivial for $s=1$, the map is the one of the  module structure.
 We will show below that the composite:
 $$
 L_{L_2} \otimes \Gamma^{p-1}(V^\#) \hra  L^0_2 \otimes L^1_p \otimes V^\# \ra  L^0_2  \otimes V^\#
 $$
  takes values into $L_{L^0_2}\otimes V^\#$. We will show also it factors through nontrivial the tensor products. As there are no nontrivial morphism from a nontrivial tensor product to $\Gamma^1$ (see in \cite{FLS},  " Alg\`ebre de Steenrod, modules instables et foncteurs polynomiaux")  the morphism will be shown to be trivial because of the module  structure.

  It follows, using the multiplicative structure on the bar resolution, that as an algebra, we have a natural isomorphism in $V$ :
  $$
  	 \left\{T_V(\Tor^{-s}_{\rmH^*(K_p)}(\rmH^*(K_2), \bF2))^s \right\}_{s \geq 0}
 \cong  \left\{L^0_2 \otimes \Lambda^s(\Gamma^{p-1}) \right\}_{s \geq 0} \;\;\;\; .
  	$$

This identifies with the $E^\infty$-term on the diagonal $p+q=0$ and therefore is the associated graded of $f(\sqrt{\rmH^*(E))}$,  \ref{conjexe} (a) follows.  Let us make this a bit more precise: applying the functor $T_V$
to the Eilenberg-Moore spectral sequence yields the Eilenberg-Moore spectral sequence for the mapping spaces as observed in \cite{FS}.

 \bigskip

In order to complete the  proof one needs to
 describe how to  compute  the induced morphism $T_V(\psi)$ and gets some  informations about the module structure of $T_V(\rmH^*(K_2))$ over $T_V(\rmH^*(K_p))$. We will make this a bit longer than necessary for future use.

We will  show  two facts:
\begin{itemize}
\item the first one is that  $T_V(\psi)$ writes  as $T_V(\psi)^0 \otimes T_V(\psi)^{>0}$,
\item the second
 one is that $T_V(\psi)^{1} : \Gamma^{p-1} (V^\#) \otimes F(1)^1 \ra  V^\# \otimes F(1)^1 $ under the hypothesis $\psi$ is decomposable, factors through (a direct sum of) nontrivial tensor products $A \otimes B \otimes F(1)^1$. As there are no nontrivial morphisms (\cite{FLS}) from a nontrivial tensor product $A \otimes B$ to the identity functor we are done.
 \end{itemize}

  In the discussion below, we check the first affirmation  in all cases: stable classes, sums of classes, products of classes. This is done by inspection.  For the second fact one just needs to look at the argument about product of classes.

 \bigskip

In degree $0$ the morphism is induced by composition of maps: $\psi$ induces a  set map, natural in $V$, $S^2(V^\#) \ra S^p(V^\#)$. To a (set) map $ a : S^p(V^\#) \ra \bF2$ one associates
$a \circ \psi$.

\medskip

We discuss now the behavior in non zero degrees.

If $\psi$ is a stable class, it is induced by a morphism $F(p) \ra F(2)$  which we will denote also by $\psi$,
({\it i.e.} equivalently $\psi = \theta (\iota_2)$ where $\theta$ is a Steenrod operation). The induced morphism $\rmH^*(K_p) \ra \rmH^*(K_2)$
is obtained by applying the Steenrod-Epstein functor $U$ to $\psi$.

 The morphism $T_V(\psi )): T_V(\rmH^*(K_p) \ra T_V(\rmH^*(K_2)$ is obtained by applying $U$ to:
$$
T_V(F(p)) \cong \bigoplus_{0 \leq i \leq p} \Gamma^i(V) \otimes F(p-i)
\ra T_V(F(2)) \cong \bigoplus_{0 \leq j \leq 2} \Gamma^j(V) \otimes F(2-j)
$$

As there are no nontrivial morphisms to (resp. from)  $F(0)$ from (resp. to) $F(n)$), $n>0$, this morphism splits up as a direct sum
$T_V(\psi)^{0} \oplus T_V(\psi)^{>0}$.

Next, as there are no nontrivial morphism from $F(k)$ to $F(\ell)$
if $k< \ell$ the morphism $T_V(\psi)^{>0}$  has a triangular matrix in the following sense: the morphisms from  $\Gamma^i(V) \otimes F(p-i)$ to $\Gamma^j(V) \otimes F(2-j)$
can be nontrivial only $p-i \geq 2-j$.

\bigskip

The following steps are  to describe the situation on a sums and  products of classes.

Let $\psi \in \rmH^p (K_2)$, suppose that $\psi =\alpha+ \beta$,
$\alpha \in \rmH^p(K_2)$, $\beta \in \rmH^p(K_2)$. The morphism $$T_V(\psi): T_V(\rmH^*(K_p)) \ra T_V(\rmH^*(K_2))$$

 is the composite:
	$$
	T_V(\rmH^*(K_p)) \ra T_V(\rmH^*(K_p) \otimes T_V(\rmH^*(K_p))
	 \buildrel{T_V(\alpha \otimes T_V(\beta)}\over{\lra} \dots
	 	$$
	$$
	\ldots
	\buildrel{T_V(\alpha) \otimes T_V(\beta)}\over{\lra}T_V(\rmH^*(K_2) \otimes T_V(\rmH^*(K_2))
	\buildrel \mu \over \ra T_V(\rmH^*(K_2))
	$$
	where $T_V(\rmH^*(K_p)) \ra T_V(\rmH^*(K_p) \otimes T_V(\rmH^*(K_p))$ is obtained by applying $T_V \circ U$
	to $F(p) \buildrel diag \over \lra F(p) \oplus F(p)$.

\medskip

The key step is the one of a product  of classes.
Let $\psi \in \rmH^p (K_2)$, suppose that $\psi =\alpha\beta$,
$\alpha \in \rmH^u(K_2)$, $\beta \in \rmH^v(K_2)$, $u,v >0$. The morphism $$T_V(\psi): T_V(\rmH^*(K_p)) \ra T_V(\rmH^*(K_2))$$

 is the composite:
	$$
	T_V(\rmH^*(K_p)) \ra T_V(\rmH^*(K_u) \otimes T_V(\rmH^*(K_v))
	\buildrel{T_V(\alpha \otimes T_V(\beta)}\over{\lra}T_V(\rmH^*(K_2) \otimes T_V(\rmH^*(K_2))
	\buildrel \mu \over \ra \rmH^*(K_2)
	$$

where $T_V(\rmH^*(K_p)) \ra T_V(\rmH^*(K_u) \otimes T_V(\rmH^*(K_v))$ is obtained, using the universal property of $U$, from $F(p) \ra F(u) \otimes F(v) \hra U(F(u)) \otimes U(F(v))$.

\bigskip

Finally checking each step one gets:

\begin{prop} For any class $\psi$ the morphism $T_V(\psi)$ is the tensor product of $T_V(\psi)^0$ and of $T_V(\psi)^{>0}$

\end{prop}
\begin{cor}
	The morphism
$T_V(\psi)^1$ identifies is the tensor product of $T_V(\psi)^0$ and  $\Gamma^{p-1}(V^\#) \otimes F(1) \ra V^\# \otimes F(1)$.
\end{cor}

Consider now the induced morphism on $\Gamma^{p-1}(V^\#) \otimes F(1)^1 \ra V^\# \otimes F(1)^1$
By construction it factors as:

      \[
     \begin{tikzcd}
    \Gamma^{p-1}(V^\#) \otimes F(1)^1 \arrow[d]\\
    \Gamma^{u}(V^\#) \otimes F(1)^1  \otimes \Gamma^{v-1}(V^\#) \otimes F(1)^0
\oplus   \Gamma^{u}(V^\#) \otimes F(1)^0  \otimes \Gamma^{v-1}(V^\#) \otimes F(1)^1\ar[d]\\
 V^\# \otimes F(1)^1
     \end{tikzcd}
        \]
with $u+v=p$.
As $u,v>1$ in our cases the second technical result follows.

This proves part (a) of the conjecture.

\bigskip

In order to prove the strong form in  \ref{conjexe} it is enough to show that the $-1$ step of the
filtration of $f(\rmH^*(E))$, $F_{-1}$, is isomorphic to
$$\tilde  L^0_2 \otimes \Gamma^{p-1}(V^\#) \oplus \Gamma^{p-1}(V^\#)
\oplus L^0_2$$

Indeed, if so the direct summand $\Gamma^{p-1}(V^\#)$ generates multiplicatively an exterior algebra $\Lambda^*(\Gamma^{p-1}(V^\#))$, and thus the graded associated to $\sqrt {\rmH^*(E)}$ contains a subfactor  $\Lambda^*(F(p-1))$ generated on the $-1$-column  of the spectral sequence. The result follows.

Saying it differently it means that a certain power $\iota^{2^k}_{p-1}$ is not only in the image of $\rmH^*(E) \ra \rm \rmH^*(K_{p-1})$, but
that there is a lifting $x$ such that the map $$\frac {\calA_2. x}{Nil} \ra \calA_2 .\iota^{2^k}_{p-1} \subset F(p-1)$$ is an isomorphism.

\begin{exe} Such a lifting may be very hard to find. In the case of the class $\psi = \iota_2 Q_1(\iota_2)$, the class
$\Sq^4\iota_5 \otimes 1+ \Sq^2\iota_5 \otimes \iota_2 + \iota_5 \otimes \iota_2^2 $
 is a candidate.
\end{exe}

 One knows there is a short exact sequence:

$$
0 \ra  L^0_2 \ra F_{-1} \ra L^0_2 \otimes \Gamma^{p-1}(V^\#) \cong  \Gamma^{p-1}(V^\#) \oplus \overline{ L^0_2 }\otimes \Gamma^{p-1}(V^\#) \ra 0
$$
and one shows that it splits.

Because of the multiplicative structure it is enough to show that,
\begin{eqnarray}
\Ext^1_\calF( \Gamma^{p-1}, L^0_2)=0
\end{eqnarray}

As is shown below this is true in all  the cases we consider, however it is hopeless to expect such a condition to be true in full generality.  We will show that the corresponding extension group is non trivial if $p-1$ is odd
in the case of some fiber $E$ of a map from $K_3$ to $K_{p-1}$.

\medskip

There is an isomorphism (\cite{FFPS} th\'eor\`eme 12.13 page 95):
$$
\Ext^1_\calF( \Gamma^{p-1}, L^0_2) \cong (\ell^1 \circ m(L^0_2))^{p-1}
$$
 $m$ is the right  adjoint of $f$, $\ell = m \circ f$, $\ell= m \circ f$, $\ell$ is left exact and $\ell^1$ is the first right  derived functor.

The work of G. Powell \cite{Po2} shows all  the of  $L^0_2 $ in our list have a finite filtration
whose quotients are in the list below. Therefore it is enough to check (5.3)
for the functors:
\begin{itemize}
	\item the constant functor $\bF2$,
	\item $I :V \mapsto \bF2^{V^\#}$;
	\item $I_2 : V \mapsto (\bF2^{V^\#})^{\otimes 2}$;
	\item their direct summands;
	\item $B_2 :V \mapsto ((\bF2^{V^\#})^{\otimes 2})^{\bZ/2}$, note that $B_2 \cong S^2(I)$;
	\item the kernel  $\bar B_2$ of   the epimorphism $B_2 \ra I$;
	\item $D_2 :V \mapsto ((\bF2^{V^\#})^{\otimes 2})^{\GL_2(\bF2)}$;
	\item the kernel  $\bar D_2$ of   the epimorphism $D_2 \ra I$.

\end{itemize}

The first three ones are injective.

For the two following cases, one has  $m(B_2) \cong \bF2[w_1,w_2] \cong  \bF2[u,v]^{\bZ/2}$. The unstable module
$\bF2[w_1,w_2] $ has an injective resolution which starts  like:
$$
\bF2[w_1,w_2]\hra  \bF2[u,v] \buildrel{1+ \tau^* } \over \ra  \bF2[u,v] \buildrel{1+ \tau^* } \over \ra  \bF2[u,v] \oplus  \bF2[u] \ra \ldots
$$
$\tau$ is the transposition.
In order to prove this one uses the exact sequence of functors:
 $$
 0 \ra \Gamma^2 \ra (\Gamma^1)^{\otimes 2} \ra (\Gamma^1)^{\otimes 2} \ra S^2
 $$
 applied to $ \bF2[u] \cong\rmH^*(V_1)$, with
  $ \bF2[u,v] \cong\rmH^*(V_1)^{\otimes 2}$ and
$\bF2[w_1,w_2] \cong\Gamma^2(\rmH^*(V_1))	$. The morphism $(\Gamma^1)^{\otimes 2} \ra (\Gamma^1)^{\otimes 2} $
being the norm map.

 From the definition of $\ell^1$ it follows that $\ell^1(m(B_2))=0$ and the result follows for $B_2$. As $B_2$ is the direct
sum of $\bar B_2$ and $I$ result holds for $\bar B_2$.

\bigskip

This gives us the proof of Propositions \ref{case1} and \ref{case3}.

\bigskip

For the two last cases one needs to study $D_2$ and $\bar D_2$. It is done as follows, the  subfunctor
$$A_2 :V \mapsto ((\bF2^{V^\#})^{\otimes 2})^{\bZ/3} $$
is a direct summand in $\subset \bF2^{V^\#})^{\otimes 2} \cong \bF2^{\Hom(V, \bF2^{\oplus 2})})$. The action of ${\bZ/3}$ is {\textit{ via}} the inclusion into $\GL_2(\bF2)$.
It is known (\cite{Po1}) to be inserted in a short exact sequence:
$$
E \hra A_2 \era I
$$
where
$$
\bar D_2 \hra E \era \bar D_2
$$
and one has also
$$
D_2 \hra A_2 \era \bar D_2 \;\;\;\; \; .
$$
in terms of unstable modules one gets corresponding sequences applying the functor $m$,. As $m$ they may not be exact on the right, by inspection it is not in the first case and it is in the the two other cases. In the first case :
$$
m(E) \hra \bF2[u,v]^{\bZ/3} \ra \bF2[u]
$$
the image on the right being $ \bF2[u^2]$,
and
$$
\bar D(2) \hra m(E) \era \bar D(2) \;\;\;\; ,
$$
and
$$
D(2) \hra   \bF2[u,v]^{\bZ/3}  \era \bar D(2) \;\;\;\; ,
$$
$\bar D(2)$ being the kernel of  the unique nontrivial morphism $D(2) \ra \bF2[u]$.
It follows that:
$$
\ell^1(m(E)) \cong \bF2[u]/\bF2[u^2] \;\;\;\; ,
$$
and that:

$$
\ell^1(m(\bar D_2)) \cong \bF2[u]/\bF2[u^2]
$$

In our examples we are first concerned with $D(2)$. In this case we can exploit the short exact sequence:
 which, as $\bar D(2)$ is $\nil$-closed,  yields immediately  $ \ell^1\circ m( D(2))=0 $. This is \ref{case2}. There remains case \ref{case4}, keeping the notations one has a short exact sequence of unstable algebras:
 $$
 M_2 \hra D(2) \oplus H_2 \era \bF2[u^2]
 $$

As $\ell^1(D(2))= \ell^1(H_2)=0$  it impies that:
$$
\Ext^1_\calF( \Gamma^{p-1}, \bar D_2) \cong \ell^1  \circ m(\bar D(2))^{p-1}\cong  \bF2[u]/\bF2[u^2]
$$

and the strong form of the conjecture is true if $p-1$ is odd.

\begin{rmq} As $\bar D_2$ is not a quotient  of $\rmH^*(K_2)$, the bottom class in degree $3$, it does not appear (alone)
in our list. However it is a quotient of $\rmH^*(K_3)$, thus in such a case the proof of the strong form of the conjecture would not follow
because
$$
\Ext^1_\calF( \Gamma^{p-1}, \bar D_2) \cong \ell^1  \circ m(\bar D(2))^{p-1} \cong  \bF2[u]/\bF2[u^2]
$$
is trivial only if $p-1$ is odd.
\end{rmq}

\bibliographystyle{plain}
\bibliography{bibli}

\end{document}